\documentclass[11pt,reqno]{amsart}
  \usepackage{latexsym} 
  \usepackage[all]{xy}
  \usepackage{amsfonts} 
  \usepackage{amsthm} 
  \usepackage{amsmath} 
  \usepackage{amssymb}
  \usepackage{pifont}  
  \usepackage{enumerate}
  \usepackage{dcolumn}
  \usepackage{comment}
  \usepackage{hyperref}  
\newcolumntype{2}{D{.}{}{2.0}}
  \xyoption{2cell}

 \usepackage{tikz}
\usetikzlibrary{arrows}

   \def\<{{\langle}} 
  \def\>{{\rangle}}

  \def\note#1{{}}

  \def\note#1{}

  \def\beq{\begin{equation}} 
  \def\eeq{\end{equation}}

  \newcounter{zlist}

  \newcounter{blist}

  \newcounter{rlist}

\def\stac#1{\raise-.2cm\hbox{$\stackrel{\displaystyle\otimes}{\scriptscriptstyle{#1}}$}}

\def\cten#1{\raise-.2cm\hbox{$\stackrel{\displaystyle\reallywidehat{\otimes}}
{\scriptscriptstyle{#1}}$}}

  \def\Label#1{\label{#1}\ifmmode\llap{[#1] }\else 
  \marginpar{\smash{\hbox{\tiny [#1]}}}\fi} 
  \def\Label{\label}

  \newtheorem{proposition}{Proposition}[section]
  \newtheorem{lemma}[proposition]{Lemma} 
  \newtheorem{corollary}[proposition]{Corollary} 
  \newtheorem{theorem}[proposition]{Theorem} 
  
\theoremstyle{definition} 
  \newtheorem{definition}[proposition]{Definition}
    
  \newtheorem{example}[proposition]{Example}

  \theoremstyle{remark} 
  \newtheorem{remark}[proposition]{Remark}

  \newcounter{c} 
  \renewcommand{\[}{\setcounter{c}{1}$$} 
  \newcommand{\etyk}[1]{\vspace{-7.4mm}$$\begin{equation}\Label{#1} 
  \addtocounter{c}{1}} 
  \renewcommand{\]}{\ifnum \value{c}=1 $$\else \end{equation}\fi} 
   \numberwithin{equation}{section}

\def\ZZ{{\mathbb Z}}

\newcommand{\qQ}{\mathrm{Q}}

\newcommand{\tT}{\mathrm{T}}

\newcommand{\Cc}{\mathcal{C}}

\newcommand{\tz}[2]{\tau^{z}_{#1}(#2)}

\newcommand{\szp}[2]{\sigma^{p}_{#1}(#2)}
\newcommand{\tzp}[2]{\tau^{p}_{#1}(#2)}

\def\*C{{}^*\hspace*{-1pt}{\Cc}}

\def\text#1{{\rm {\rm #1}}}

 \def\h{\mathbf{h}}

 \def\1{\mathbf{1}}

\def\1\mathbf{1}

\def\|#1{\overline{#1}}

\def\h#1 {\hat{#1}}
\usepackage{scalerel,stackengine}
\stackMath
\newcommand\reallywidehat[1]{%
\savestack{\tmpbox}{\stretchto{%
  \scaleto{%
    \scalerel*[\widthof{\ensuremath{#1}}]{\kern.1pt\mathchar"0362\kern.1pt}%
    {\rule{0ex}{\textheight}}
  }{\textheight}%
}{2.4ex}}%
\stackon[-6.9pt]{#1}{\tmpbox}%
}
\begin{document}

\title[Near braces and $p$-deformed braided groups]{Near  braces and $p$-deformed braided groups}

\begin{abstract}
\end{abstract}

\author{Anastasia Doikou}
\author{Bernard  Rybo{\l}owicz}

\address{(A.Doikou $\&$ B. Rybo{\l}owicz)
Department of Mathematics, 
Heriot-Watt University, 
Edinburgh EH14 4AS,  and Maxwell Institute for Mathematical Sciences, Edinburgh, UK}

\email{A.Doikou@hw.ac.uk,\ B.Rybolowicz@hw.ac.uk}


\baselineskip=15pt
\date\today

\vskip 0.4in

\begin{abstract}
\noindent  Motivated by recent findings on the derivation of parametric non-involutive solutions of the Yang-Baxter equation 
we reconstruct the underlying algebraic structures, called near braces.  Using the notion of the near braces we 
produce  new multi-parametric, non-degenerate,  non-involutive solutions of the set-theoretic Yang-Baxter equation. 
These solutions are generalisations of the known ones coming from braces and skew braces. Bijective maps associated to the inverse solutions 
are also constructed.  Furthermore,  we introduce the generalized notion of $p$-deformed braided groups and 
$p$-braidings and we show that every $p$-braiding is a solution of the braid equation.  
We also show that certain multi-parametric maps within the near braces provide special cases of $p$-braidings.
\end{abstract}

\maketitle

\date{}
\vskip 0.2in



\section{Introduction}

\noindent The aim of the present study is two-fold: on the one hand, motivated by recent fidings on parametric solutions \cite{DoiRyb} 
of the set-theoretic \cite{Drin, ESS} Yang-Baxter equation (YBE) \cite{Baxter, Yang} we derive the underlying algebraic structure associated to these solutions. On the other hand using the derived algebraic frame we introduce novel  multi-parametric classes of solutions of the YBE.

It is well established now that braces, first introduced by Rump \cite{[26]}, describe  all
non-degenerate involutive solutions of the YBE,  whereas skew braces were later introduced to describe non-involutive, non-degenerate solutions 
of the YBE \cite{GV}.  Indeed,  based on the ideas of \cite{[26]} and \cite{GV} and on recent findings regarding
parametric solutions of the set-theoretic YBE \cite{DoiRyb} we construct the generic algebraic structure, called near brace,  
that provides 
solutions to the set-theoretic braid equation.  Moreover, motivated by the definition of the braided group \cite{chin} 
and the work of \cite{GatMaj}, we introduce an extensive definition of a $p$-deformed braided group and $p$-braidings, 
which are solutions of the set-theoretic braid equation. All the parametric solutions derived here are indeed $p$-braidings.
It is worth noting that  the study of solutions of the set-theoretic Yang-Baxter equation and the associated algebraic structures 
have created a particularly active new field during the last decade or so
(see for instance \cite{Ba, bcjo, CJV, CMS, [6],fc}). The key observation is that by relaxing more and more conditions on the underlying algebraic structures one identifies more general classes of solutions (see e.g. \cite{Catino, CMS, JKA, kava, KSV, Lebed, SVB}, \cite{GI3}--\cite{gateva}). 
It is also worth noting that interesting links with quantum integrable systems \cite{DoiSmo1, DoiSmo2} as well the 
quasi-triangular quasi-bialgebras \cite{Doikoutw}-\cite{DoiRyb} have been recently established,  opening up new intriguing  paths of investigations.

We briefly describe what is achieved in this study, and in particular what are the findings  in each section. 
In the remaining of this section we review some necessary 
ideas on non-degenerate set-theoretic solutions of the YBE and the associated algebraic structures, i.e.  braces and skew braces.
In Section 2 inspired by the parametric solutions of the YBE introduced in \cite{DoiRyb} we reconstruct the generic 
associated algebraic structure called near brace.  In fact, every near  brace can turn to a skew 
brace by defining a suitably modified (deformed) addition; this is described in Theorem \ref{theorem1}.  
The key idea is to simultaneously consider $\check r$ and its inverse given that we are exclusively interested in non-degenerate solutions of the braid equation. Having derived the underlying algebraic structure
we move to Subsection 2.1 where we extract multi-parametric bijective maps and 
hence to  identify non-degenerate,  multi-parametric solutions of the YBE as well as their inverses.  
In Subsection 2.2 we provide a generalized definition of the braided group and braidings 
($p$-braidings, $p$ stands for parametric) 
by relaxing some of the conditions appearing in the definition of \cite{chin} (see also relevant findings in \cite{GatMaj}.)  
Furthermore,  we show that the generalized $p$-braidings are non-degenerate solutions of the YBE 
and the bijective maps coming for the near braces 
provide automatically $p$-braidings.

\subsection*{Preliminaries} Before we start our analysis and present our findings in the subsequent section we review below basic 
preliminary notions relevant to our investigation.  Specifically, we recall the problem of solving the set-theoretic braid equation 
and some fundamental results. Let $X=\{x_{1}, \ldots x_n\}$ be a set and ${\check r}^z:X\times X\rightarrow X\times X,$ 
where $z \in X$ is a fixed parameter, first introduced in \cite{DoiRyb}.
We denote 
\begin{equation}
{\check r}^z(x,y)=(\sigma^z _{x}(y), \tau^z _{y}(x)). \label{setY}
\end{equation}
We say that $\check r^z$ is non-degenerate if 
$\sigma^z _{x}$ and $\tau^z _{y}$ 
are bijective maps, and
$(X, {\check r})$ is a set-theoretic solution of the braid equation if
\begin{equation}
({\check r}^z\times \mbox{id})( \mbox{id} \times {\check r}^z)({\check r}^z\times \mbox{id})=(\mbox{id}\times 
{\check r}^z)({\check r}^z\times \mbox{id})(\mbox{id}\times {\check r}^z).\label{braid}
\end{equation}
The map $\check r$ is called involutive if $\check r^z \circ \check r^z=\mbox{id}.$

We also introduce the map $r: X\times X\rightarrow X\times X,$ such that $r^z = \check r^z\pi,$ 
where $\pi: X\times X\rightarrow X\times X$ 
is the flip map: $\pi(x,y) = (y,x).$ Hence, $r^z(y,x) =( \sigma_x^z(y), \tau_y^z(x)),$ and it satisfies the YBE:
\begin{equation}
r^z_{12}\ r^z_{13}\  r^z_{23} =r^z_{23}\ r^z_{13}\ r^z_{12},  \label{YBE}
\end{equation} 
where we denote $r^z_{12}(y,x,w) = (\sigma_x^z(y), \tau_y^z(x),w),$ $r^z_{23}(w,y,x) = (w, \sigma_x^z(y), \tau_y^z(x))$ 
and \\  $r^z_{13}(y,w,x) = (\sigma_x^z(y), w,\tau^z_y(x)).$

We review below the constraints arising by requiring $(X, \check r^z)$ to be a solution 
of the braid equation (\cite{Drin,  ESS, [25], [26]}). Let,
\[({\check r}^z\times \mbox{id})(\mbox{id}\times {\check r}^z)({\check r}^z\times \mbox{id})(\eta, x, y)=(L_1, L_2, L_3),\]
\[(\mbox{id}\times {\check r}^z)({\check r}^z\times \mbox{id})(\mbox{id}\times {\check r}^z)(\eta,x,y)= (R_1, R_2, R_3),\]
where, after employing expression  (\ref{setY}) we identify:
\begin{eqnarray}
L_1 = \sigma^z_{\sigma^z_{\eta}(x)}(\sigma^z_{\tau^z_{x}(\eta)}(y)),\quad L_2 = 
\tau^z_{\sigma^z_{\tau^z_x(\eta)}(y)}(\sigma^z_{\eta}(x)), \quad L_3 =\tau^z_y(\tau^z_x(\eta)),\nonumber
\end{eqnarray}
\begin{eqnarray}
R_1=\sigma^z_{\eta}(\sigma^z_x(y)), \quad R_1=\sigma^z_{\tau^z_{\sigma^z_x(y)}(\eta)}(\tau^z_{y}(x)), \quad 
R_3= \tau^z_{\tau^z_{y}(x)}(\tau^z_{\sigma^z_{x}(y)}(\eta)). \nonumber
\end{eqnarray}
And by requiring $L_i =R_i,$ $i\in \{1,2,3\}$ we obtain the following fundamental constraints for the associated maps:
\begin{eqnarray}
&&  \sigma^z_{\eta}(\sigma^z_x(y))= \sigma^z_{\sigma^z_{\eta}(x)}(\sigma^z_{\tau^z_{x}(\eta)}(y)), \label{C1}\\
&& \tau^z_y(\tau^z_x(\eta)) = \tau^z_{\tau^z_{y}(x)}(\tau^z_{\sigma^z_{x}(y)}(\eta)), \label{C2}\\
&&  \tau^z_{\sigma^z_{\tau^z_x(\eta)}(y)}(\sigma^z_{\eta}(x))= \sigma^z_{\tau^z_{\sigma^z_x(y)}(\eta)}(\tau^z_{y}(x)). \label{C3}
\end{eqnarray}
Note that the constraints above are the ones of the set-theoretic solution (\ref{setY}), given that $z$ 
is a fixed element of the set,  i.e.  for different elements $z$ we obtain in principle distinct solutions of the braid equation.

We review now the  basic definitions of the algebraic structures that provide set-theoretic solutions of the braid equation, 
such as left skew braces and  braces.  We also present some key properties associated 
to these structures that will be useful when formulating some of the main findings of the present study, summarized in Section 2.
\begin{definition}[\cite{[25], [26], [6]}]
A {\it left skew brace} is a set $B$ together with two group operations $+,\circ :B\times B\to B$, 
the first is called addition and the second is called multiplication, such that for all $ a,b,c\in B$,
\begin{equation}\label{def:dis}
a\circ (b+c)=a\circ b-a+a\circ c.
\end{equation}
If $+$ is an abelian group operation $B$ is called a {\it left brace}.
Moreover, if $B$ is a left skew brace and for all $ a,b,c\in B$ $(b+c)\circ a=b\circ a-a+c \circ a$, then $B$ is called a 
{\it skew brace}. Analogously if $+$ is abelian and $B$ is a skew brace,  then $B$ is called a {\it brace}.
\end{definition}

\begin{remark}
In the literature often left brace is just called a brace and left skew brace is called a skew brace. 
In that case various authors call skew braces two-sided skew braces.
\end{remark}

The additive identity of a left skew brace $B$ will be denoted by $0$ and the multiplicative identity by $1$. 
In every left skew brace $0=1$.  Indeed, this is easy to show: 
\begin{eqnarray}
 a\circ b = a\circ(b+0)\ \Rightarrow\ a\circ b = a\circ b -a +a\circ 0 \  \Rightarrow\ a \circ 0 = a\ \Rightarrow\  0=1. \nonumber
\end{eqnarray}

The two theorems that follow concern the case wher the parameter $z=1$.
Rump showed the following  theorem for  involutive set-theoretic solutions.
\begin{theorem}\label{Rump}(Rump's theorem, \cite{[25], [26]}).  
Assume  $(B, +, \circ)$ is a left brace. If the map  $\check r_B: B\times B \to B \times B$ is defined as 
${\check r}_B(x,y)=(\sigma _{x}(y), \tau _{y}(x))$, where $\sigma _{x}(y)=x\circ y-x$, 
$\tau _{y}(x)=t\circ x-t$, and $t $ 
is the inverse of $\sigma _{x}(y)$ in the circle group $(B, \circ ),$ then  $(B, \check r_B)$ is an involutive,  
non-degenerate solution of the braid equation.\\
Conversely,  if $(X,\check r)$ is an involutive, non-degenerate solution of the braid equation, then there exists a left brace $(B,+, \circ)$ 
(called an  underlying brace of the solution $(X, \check r)$) such that $B$ contains $X,$ $\check r_B(X\times X )\subseteq X \times X,$
and the map $\check r$ is equal to the restriction of $\check r_B$ to $X \times X.$ Both the additive $(B, +)$ 
and multiplicative $(B,\circ)$ groups of the left brace $(B,+, \circ)$ are generated by $X.$
\end{theorem}

\begin{remark} [Rump]\label{nilpotent}
 Let $(N, +, \cdot)$ be an associative ring.  If for $a,b\in N$ we define 
 \[a\circ b=a\cdot b+a+b,\] then $(N, +, \circ )$ is a brace if and only if $(N, +, \cdot)$ is a radical ring.
\end{remark}

Guarnieri and Vendramin \cite{GV}, generalized Rump's result to left skew braces and non-degenerate,  non-involutive solutions.
\begin{theorem}[{\it Theorem \cite{GV}}]\label{thm:GV}
Let $B$ be a left skew brace, then the map $\check{r}_{GV}:B\times B\to B\times B$ given for all $ a,b\in B$ by
$$
\check{r}_{GV}(a,b)=(-a+a\circ b,\ (-a+a\circ b)^{-1}\circ a\circ b)
$$
is a non-degenerate solution of set-theoretic YBE.  
\end{theorem}

\section{Set-theoretic solutions of the YBE and near braces}

\noindent  In this section starting  from a generic 
$z$-parametric set-theoretic solution of the YBE \cite{DoiRyb} we reconstruct the underlying algebraic structure, which is similar to a skew brace.
Indeed,  we introduce in what follows suitable algebraic structures that satisfy 
the fundamental constraints (\ref{C1})-(\ref{C3}),  i.e.  provide solutions of the braid equation 
and generalize the findings of Rump and Guarnieri $\&$ Vendramin.  
The following generalizations are greatly inspired by recent results in \cite{DoiRyb}.

For the rest of the subsection we consider $X$ to be a set with an arbitrary group operation 
$\circ :X\times X\to X$,  with a neutral element $1\in X$ and an inverse $x^{-1} \in X,$ 
 for all $x \in X.$  There also exists a family of bijective functions indexed by $X,$  $\sigma^z_x: X \to X,$ 
such that $y \mapsto \sigma^z_x(y),$ where $z \in X$ is some fixed parameter.
We then define another binary operation $+: X\times X \to  X,$ such that
\begin{equation}
 y +x := x \circ \sigma^z_{x^{-1}}(y\circ z)\circ z^{-1} \label{mapzz}.
\end{equation}
For convenience we will omit henceforth the fixed $z\in X$ in $\sigma^z_x(y)$
and  simply write $\sigma_x(y).$
\begin{remark} \label{bi}
The operation $+$ is associative if and only if for all $x,y,c\in X$,
\begin{equation}\label{eq:tg5}
\sigma_{c^{-1}}(y\circ z^{-1}\circ\sigma_{z\circ y^{-1}}(x))=
\sigma_{c^{-1}}(y)\circ z^{-1}\circ \sigma_{(c\circ\sigma_{c^{-1}}(y)\circ z^{-1})^{-1}}(x).
\end{equation}
\end{remark}
\noindent From now on we will assume that the operation $+$ is associative, that is condition \eqref{eq:tg5} holds.

Also, we recall that we focus only on non-degenerate, invertible solutions $\check r.$  
Given that $\sigma_x$ and $\tau_y$ are bijections the inverse maps also exist such that
\begin{equation}
\sigma^{-1}_x(\sigma_x(y)) = \sigma_x(\sigma^{-1}_x(y))= y ,  \quad  \tau^{-1}_{y}(\tau_y(x)) = \tau_{y}(\tau^{-1}_y(x)) =x \label{inv1}
\end{equation}
Let
the inverse $\check r^{-1}(x,y) = (\hat \sigma_x(y), \hat \tau_y(x))$ exist with $\hat \sigma_x,\ \hat \tau_y$ 
being also bijections,  that satisfy:
\begin{equation}
\sigma_{\hat \sigma_x(y)} (\hat \tau_y(x)) = x=\hat \sigma_{\sigma_x(y)} (\tau_y(x)), \quad 
\tau_{\hat \tau_y(x)}(\hat \sigma_x(y)) = y = \hat \tau_{\tau_y(x)}(\sigma_x(y)).   \label{ide1}
\end{equation}
Taking also into consideration (\ref{inv1}) and (\ref{ide1})  and that  $\sigma_x,\tau_y$ 
and $\hat \sigma_x,  \hat\tau_y$ are bijections, we deduce:
\begin{equation}
\hat \sigma^{-1}_{\sigma_x(y)}(x) = \tau_y(x),  \quad \hat \tau^{-1}_{\tau_y(x)}(y) = \sigma_x(y).
\end{equation}

We assume that the map $\hat \sigma$ appearing in the inverse matrix $\check r^{-1}$ has the general form
\begin{equation}
\hat \sigma_{x}(y) :  =  x \circ (x^{-1} \circ z_2 + y\circ z_1) \circ   \xi,  \label{mapzz2}
\end{equation}
where  the parameters $z_{1,2}, \ \xi$ are to be identified.  
The derivation of $\check r$ goes hand in hand with the derivations of $\check r^{-1}$
(see details in \cite{DoiRyb} and later in the text when deriving a generic $\check r$ and its inverse).  
In the involutive case the two maps coincide and $x+y = y+x.$ However, 
for any non-degenerate,  non-involutive solution 
both bijective maps $\sigma_x, \hat \sigma_x$ should be 
considered together with the fundamental conditions (\ref{ide1}).

We present below a series of useful Lemmas that will lead to one of our main theorems.
\begin{remark} \label{rr1} This is just a reminder of a well known fact. 
We recall that $\sigma_x$ is a bijective function.  Recalling also definiton (\ref{mapzz}):
\begin{equation}
\sigma_x(y_1) = \sigma_x(y_2) \Leftrightarrow   y_1 \circ z^{-1}  + x^{-1} = y_2 \circ z^{-1}  + x^{-1},
\end{equation}
which implies right cancellation of $+.$ Similarly $ \hat \sigma_x$ is 
a bijective function and this leads to left cancellation.
\end{remark}

\begin{lemma}\label{lem:bij}
For all $x\in X$, the operations $+x,\ x+:X\to X$ are bijections.
\end{lemma}
\begin{proof}
Let $y_1,y_2\in X$ be such that $y_1+x=y_2+x$, then
$$
x\circ \sigma_{x^{-1}}(y_1\circ z)\circ z^{-1}=x\circ \sigma_{x^{-1}}(y_2\circ z)\circ z^{-1}\implies  \sigma_{x^{-1}}(y_1\circ z)=\sigma_{x^{-1}}(y_2\circ z),
$$
since $\circ$ is a group operation and $\sigma_{x^{-1}}$ is injective, we get that $y_1=y_2$ and $+x$ is injective for any $x\in X$. 
From the surjectivity, we observe that since $\sigma_{x^{-1}}$ is bijective, we can consider $d=\sigma^{-1}_{x^{-1}}(x^{-1}\circ c\circ z)\circ z^{-1}$, 
one can easily see that $d+x=c$, and since $c$ is arbitrary we get that $+x$ is a surjection. Thus $+x$ is a bijection. 
Similarly,  from the bijectivity of $\hat \sigma_x$ and (\ref{mapzz2}) we show that $x+$ is also a bijection.
\end{proof}

We now introduce the notion of neutral elements in $(X,+)$

\begin{lemma} Let  $(X,+)$ be a semigroup,  then for all
$ x \in X$ there exists $ 0_x \in X$ such that $0_x + x = x.$  Moreover, for all $x, y \in X,$ $0_x = 0_y =0,$
i.e. $0$ is the unique left neutral element.  The left neutral element $0$ is also right neutral element. 
\end{lemma}
\begin{proof}
Notice that due to bijectivity of $\sigma_x$, we can consider the element 
\begin{equation}
0_x:= \sigma^{-1}_{x^{-1}}(z)\circ z^{-1}  \in X, \nonumber
\end{equation}
recall also the definition of $+$ in (\ref{mapzz}), then simple computation shows:
\begin{equation}
0_x + x = x \circ  \sigma_{x^{-1}}(\sigma^{-1}_{x^{-1}}(z)) \circ z^{-1} = x\circ z \circ z^{-1} = x.
\end{equation}

We have,
\begin{equation}
0_x  + x = x \Rightarrow 0_x  + x +y = x+y, \nonumber
\end{equation}
but also 
\begin{equation}
0_{x+y} + x +y = x+y.  \nonumber
\end{equation}
The  last two equations lead to $0_x + x+y  = 0_{x +y }+x +y,$ and due Lemma \ref{lem:bij} right cancellation holds, so we
 get that $0_x=0_{x+y}$ for all $y\in X$. Observe that by the Lemma  
\ref{lem:bij}, $x+$ is a surjection, that is for all $w\in X$ exists $y\in X$ such that $x+y=w$, that is $0:=0_x=0_w$ for all $w\in X$.

Moreover,  $0 +y =y \Rightarrow x+0 +y = x+y$ and due to associativity and 
right cancellativity (Lemma \ref{lem:bij}) we get $x+ 0 = x,$ for all $x\in X$.
\end{proof}

\begin{lemma} Let $0$ be the neutral element in $(X,+)$,  then
for all $ x \in X$ there exists $-x \in X,$ such that $-x  + x = 0$ (left inverse).  
Moreover, $-x \in X$ is a right inverse,  i.e.  $x +( -x )= 0$ $\forall x \in X.$ 
That is $(X,+,0)$ is a group.
\end{lemma}
\begin{proof} 
Observe that due to bijectivity of $\sigma_x$, we consider the element
\begin{equation}-x:=\sigma_{x^{-1}}^{-1}(x^{-1}\circ 0\circ z)\circ z^{-1} \in X.
\end{equation}
Simple computation shows it is a left inverse,
$$
-x+x=x\circ \sigma_{x^{-1}} (\sigma_{x^{-1}}^{-1}(x^{-1}\circ 0\circ z)\circ z^{-1}\circ z)\circ z^{-1}=0.
$$
By associativity we deduce that $x+(-x)+x=0+x$, we get that $x+(-x)=0$, and $-x$ is the inverse.  
\end{proof}
To conclude,  having only assumed associativity in $+$ (\ref{mapzz}) we  deduced that $(X,+)$ is a group.
We may now present our main findings described in the following central theorem.

\begin{theorem} \label{theorem1}
Let $(X, \circ)$ be a group and $\check r: X \times X \to X \times X$ be 
such that $\check r (x, y) = (\sigma_x(y), \tau_x(y))$ is a non-degenerate solution of the 
set-theoretic braid equation. Moreover,  we assume that:
\begin{enumerate}[(A)]
\item\label{prop:con0} The pair $(X,+)$ ($+$ is  defined in (\ref{mapzz})) is a group.
\item\label{prop:con1} There exists $\phi:X\to X$ such that for all $a,b,c\in X$ $a\circ (b+c)=a\circ b+\phi(a)+a\circ c.$
\item \label{prop:con2} For $h \in \{z, \xi\} \in X$ appearing in $\sigma_x(y)$  and $\hat \sigma_x(y)$ there exist $\widehat{\phi}:X\to X$ 
such that for all $a,b\in X$ $(a+b)\circ h=a\circ h+\widehat{\phi}(h)+b\circ h.$
\item The neutral element $0$ of $(X, +)$ has a left and right distributivity.
\end{enumerate}
Then for all $a,b,c\in X$ the following statements hold:
\begin{enumerate}
\item $\phi(a) = - a\circ 0$ and $~\widehat{\phi}(h) = -0\circ h$, \label{prop:main2}
\item  $\sigma_a(b)=(a\circ b\circ z^{-1}-a\circ 0+1)\circ z= a\circ b - a \circ 0\circ z + z.$\label{prop:main3}
\item $~a-a\circ 0 =1$ and  (i) $0\circ 0=-1$  (ii) $1+1 = 0^{-1}.$  \label{prop:main4} 
\item If $~z_2\circ \xi  =0^{-1},$ $~ 0 \circ \xi   = z_1 \circ \xi = z^{-1} \circ 0^{-1},$ then  (i) $\hat \sigma_a(b) \circ \hat \tau_b(a)= a\circ b  = \sigma_a(b) \circ\tau_b(a)$  (ii) $ - a\circ 0+a =1.$ \label{prop:main4b} 
\end{enumerate}
\end{theorem}
\begin{proof}

$ $

\begin{enumerate}

\item In the following the distributivity rule  $a\circ(b +c ) = a\circ b +\phi(a) + b\circ c$ holds, then
\begin{equation}
a = a\circ (0 + 1) = a \circ 0 + \phi(a) + a\circ 1\  \Rightarrow \phi(a) = -a\circ 0, \nonumber
\end{equation}
Also, for those $z \in X$ such that $(a +b )\circ z = a\circ z + \hat \phi(z) + b\circ z$ we have 
\begin{equation} 
z =(0+1) \circ z  = 0\circ z + \hat \phi(z) + z\ \Rightarrow \hat \phi(z) = -0 \circ z. \nonumber
\end{equation}

\item Using the distributivity rule we obtain
\begin{equation}
\sigma_a(b) = (a\circ b \circ z^{-1} - a \circ 0 +1 )\circ z. \label{dist}
\end{equation}
Before we move on with the rest of the proof it is useful to calculate
$(-a) \circ z,$ indeed:
\begin{eqnarray}
0 \circ z &= & (a -a ) \circ z  \Rightarrow 0\circ z = a \circ z - 0\circ z + (-a) \circ z  \nonumber\\ &  
\Rightarrow& (-a)\circ z = 0\circ z - a \circ z + 0 \circ z. 
\end{eqnarray}
The latter then leads to the following convenient identity (see also \cite{DoiRyb} and Lemma \ref{lem:ter} later in the text)
\begin{equation}
(a-b +c)\circ z = a\circ z -b\circ z + c \circ z, \nonumber
\end{equation}
and hence (\ref{dist}) becomes $~\sigma_a(b) = a\circ b - a \circ 0\circ z + z. $
\item Due to the fact that $\check r$ satisfies the braid equation we may  employ (\ref{C1}) and the general distributivity rule (see also (\ref{dist})):
\begin{eqnarray}
\sigma_a(\sigma_b(c)) &= & (a \circ \sigma_b(c) \circ z^{-1} - a\circ 0 +1 )\circ z \nonumber\\
&=& (a\circ b (c \circ z^{-1} +b^{-1}) \circ z \circ z^{-1} - a\circ 0 +1)\circ z \nonumber\\
& =& \big (a\circ b \circ c \circ z^{-1} - a\circ b\circ 0 + a - a\circ 0 +1\big )\circ z. \label{basic1} \nonumber
\end{eqnarray}
But due to condition (\ref{C1}) and by setting $c = 0\circ z,$ we deduce that
$a- a\circ 0 = \zeta,$ for all $a \in X$   ($\zeta$ is a fixed element in $X$), but for $a=1$  we immediately obtain $\zeta = 1,$
i.e.  
\begin{equation}
a -a \circ 0 =1. \label{conditiona}
\end{equation} 

\noindent (i) By setting $a=0$ in (\ref{conditiona}) we have $0\circ 0 =-1.$

\noindent (ii) $0\circ (1+1) = 0\circ 1 - 0\circ 0 + 0\circ 1 \Rightarrow 0\circ (1+1) = 1 \Rightarrow 1+1 = 0^{-1}.$\\

\item  For the following we set  $~z_2\circ \xi  =0^{-1},$ $~ 0 \circ \xi   = z_1 \circ \xi = z^{-1} \circ 0^{-1}.$
\\

\noindent (i) Recall the form of $\hat \sigma_a(b) $ (\ref{mapzz2}), and use the distributivity rules, then 
\begin{equation}
\hat \sigma_a(b) = z_2\circ \xi - a \circ 0 \circ \xi + a\circ b\circ  z_1\circ \xi.
\end{equation}

We consider now the fixed constants: $z_2\circ \xi  =0^{-1},$ $ 0 \circ \xi   = z_1 \circ \xi = z^{-1} \circ 0^{-1}.$  
Note that if $z$ satisfies the right distributivity then so does $z^{-1}$ (see Proposition 2.3 in \cite{DoiRyb}) and also $0 \circ z,$ 
given that $0$ has left and right distributivity. We recall relations (\ref{ide1}) for the maps, then
\begin{eqnarray}
& &\sigma_{\hat \sigma_a(b)}(\hat \tau_{b}(a)) = a \Rightarrow \hat \sigma_a(b) \circ \hat \tau_b(a) - \hat \sigma_a(b) \circ 0\circ z + z = a \Rightarrow \nonumber\\
& & \hat \sigma_a(b) \circ \hat \tau_b(a) - (0^{-1}-a\circ z^{-1} \circ 0^{-1}+ a\circ b \circ z^{-1} \circ 0^{-1}) \circ 0 \circ z + z =a \nonumber\\
& & \hat \sigma_a(b) \circ \hat \tau_b(a)  - a\circ b  +a -z +z =a \Rightarrow\nonumber\\
 & &  \hat \sigma_a(b) \circ \hat \tau_b(a)  = a\circ b. \nonumber
\end{eqnarray}
Similarly,  $\hat \sigma_{ \sigma_a(b)}( \tau_{b}(a)) = a \Rightarrow \sigma_a(b) \circ  \tau_b(a)  = a\circ b.$
\\

\noindent (ii)  We consider $z_2\circ \xi  =0^{-1},$ $ 0 \circ \xi   = z_1 \circ \xi = z^{-1} \circ 0^{-1},$ and consequently, as shown above, $ \sigma_a(b) \circ  \tau_b(a)  = a\circ b= \hat \sigma_a(b) \circ \hat \tau_b(a).$  We also recall condition 
(\ref{C2}) of the braid equation and $a\circ b = \sigma_a(b)\circ \tau_{b}(a),$ indeed
\begin{equation}
\tau_c(\tau_b(a)) = \sigma_{\tau_b(a)}(c)^{-1} \circ \sigma_a(b)^{-1} \circ a\circ b \circ c \nonumber
\end{equation}
and due to the form of (\ref{C2}) we conclude
\begin{equation}
\sigma_a(b) \circ \sigma_{\tau_b(a)}(c) = \sigma_{a}(\sigma_b(c)) 
\circ \sigma_{\tau_{\sigma_b(c)}(a)}(\tau_c(b)).  \label{s1} 
\end{equation}
We focus on 
\begin{eqnarray}
\sigma_a(b) \circ \sigma_{\tau_b^z(a)}(c) 
&=& \sigma_a(b) \circ (\tau_b(a) \circ c \circ z^{-1} -\tau_b(a)\circ 0 + 1)\circ z 
\nonumber\\
& =&  (a \circ b \circ c\circ z^{-1}  -a\circ b \circ 0 + \sigma_a(b)) \circ z \nonumber\\
& =& (a \circ b \circ c\circ z^{-1}  -a\circ b \circ 0 + a\circ b -a\circ 0 \circ z + z) \circ z.\label{s2}
\end{eqnarray}
Taking into consideration the form of (\ref{s1}) and (\ref{s2}) and the fact that 
$b \circ c = \sigma_c(c) \circ \tau_c(b),$ we conclude that $\forall a \in X,$ 
$-~a\circ 0 + a = \hat \zeta,$ where $\hat \zeta \in X$ is
a fixed element, and for $a=1$ we deduce that $ \hat \zeta =1,$ i.e. $-a\circ 0 + a = 1.$
\end{enumerate}
\end{proof}
\begin{remark} \label{remark2} Due to $a- a\circ 0 =-a\circ 0 +a= 1,$ $\forall a \in B,$ 
we deduce that $a+1 =1+a,$  for all $ a \in B.$ 
\end{remark}

We call the algebraic construction deduced in Theorem  \ref{theorem1}  a  {\it near brace},  in analogy to near rings, specifically:
\begin{definition}
A {\it  near  brace} is a set $B$ together with two group operations $+,\circ :B\times B\to B$, 
the first is called addition and the second is called multiplication, such that for all $ a,b,c\in B$,
\begin{equation}
a\circ (b+c)=a\circ b-a\circ 0+a\circ c,\label{cond2}.
\end{equation}
We denote by $0$ the neutral element of the $(B, +)$ group and 
by $1$ the neutral element of the $(B, \circ)$ group. 
We say that a near brace $B$ is an abelian  near brace if $+$ is abelian. 
\end{definition}

We say that a near brace $B$ is a {\em singular near brace} if for all $a\in B,$ $~a- a\circ 0 =-a\circ 0 +a= 1.$
Near braces will be particularly useful in the next subsection, where we introduce a method of finding solutions depending on multiple parameters.

\noindent In the special case where $0=1,$ we recover a skew brace.  We also show  below some useful properties for near braces.
\begin{lemma} {\cite{DoiRyb}}  \label{lem:ter}
Let $(B,+, \circ)$ be a near brace, then 
\begin{enumerate}
\item $a \circ(-b)=  a\circ 0 - a\circ b +a\circ 0. $
\item \label{lem:ter:truss}Condition \eqref{cond2} is equivalent to the following condition:
$$
\forall{a,b,c,d\in B}\  \ a\circ (b-c+d)=a\circ b-a\circ c+a\circ d.
$$
\end{enumerate}
\end{lemma}
\begin{proof}
\noindent

\begin{enumerate}
\item $a \circ (b -b) = a \circ 0 \Rightarrow a \circ b- a\circ 0 + a \circ(-b) = a\circ 0, $ 
which leads to $a\circ (-b) = a\circ 0 - a\circ b + a\circ 0.$
\item Let (\ref{cond2}) hold then 
\begin{eqnarray}
& &a\circ ( b -c + d) = a\circ (b-c) -a\circ0 +a\circ d = \cr & &
a \circ b - a\circ0 + a\circ 0 - a\circ c + a\circ 0 -a\circ0 +a\circ d =\cr & &
a\circ b - a\circ c + a\circ d.
\end{eqnarray}
Conversely,  let $a \circ (b-c+d)=a\circ b-a\circ c+a\circ d$ hold, then 
\begin{equation}
a \circ (b+c)= a \circ (b-0+c) = a\circ b-a\circ 0+a\circ c.
\end{equation}
\end{enumerate}
\end{proof}

\begin{example}
Let $(B, \circ)$ be a group with neutral element 1 and define $a +b := a \circ \kappa^{-1} \circ b,$ 
where $1\neq \kappa \in B$ is an element of the center of $(B, \circ).$ Then $(B,+,\circ)$
is a singular near brace with neutral element $0 = \kappa,$ and we call it the trivial near brace 
{\footnote{We are indebted to Paola Stefanelli for sharing this example with us.}} .
\end{example}

\begin{example}\label{ex:complex}
Let us consider the following near-truss introduced in \cite[Page 710]{BrzRybMer}:
$$
\begin{aligned}
\qQ(O(i))= \left\{ \frac{m}{2p+1}+ \frac{n}{2q+1}i\;|\; p,q\in \ZZ, \; \mbox{$m+n$ is an odd integer}\right\}.
\end{aligned}
$$
Then $\qQ(O(i))$ together with operations $a+_ib=a-i+b$ and $a\circ b=a\cdot b$ forms a near brace, where $+,\cdot$ are addition and multiplication of complex numbers, respectively.
\end{example}

\begin{definition}
Let $(B,+,\circ)$ and $(S,+,\circ)$ be near braces. We say that $f:B\to S$ is 
a near brace morphism if for all $a,b\in B$,
$$
f(a+b)=f(a)+f(b)\qquad \&\qquad f(a\circ b)=f(a)\circ f(b).
$$
\end{definition}

\begin{lemma} Let $f:X\to X$ be a map, such that for all
$a,b\in X$ $f(a\circ b-a\circ z+z)=f(a)\circ f(b)-f(a)\circ 0\circ z+z,$ 
and there is $e\in X,$ $f(e)=1.$ If such a map $f$ exists then $0=1.$
\end{lemma}
\begin{proof}
We assume that such map $f:X\to X$ exists. Then for all $a\in X$,
$$
f(z)=f(a\circ z-a\circ z+z)=f(a)\circ f(z)-f(a)\circ 0\circ z+z,
$$ 
and by setting $a=e,$ $f(e)=1$:
$$
f(z)=f(z)-0\circ z+z\implies 0\circ z=z \implies 0=1,
$$
where the last implication follows from the fact that $z$ is invertible.
\end{proof}

\begin{remark} Observe that since every bijection is surjective, the preceding Lemma states
that if $\sigma^{'}_a(b):= a\circ b-a\circ 0 \circ z+z$ gives a solution of YBE,  it is isomorphic to the
solution given by $\sigma_a(b):= a\circ b-a \circ z+z$ if $0 = 1, $ that is near skew brace is a left skew
brace.
\end{remark}


\begin{corollary}
Let $B$ be near brace. Then by the Lemma \ref{lem:ter} \eqref{lem:ter:truss}, the triple $\tT(B):=(B,[-,-,-],\circ)$, where for all $a,b,c\in B,$ $[a,b,c]:=a-b+c,$ is a near-truss such that $(B,\circ)$ is a group. Thus the triple $(B,+_1,\circ)$, where $a +_1 b := a-1+b$ for all $a,b \in B,$ is a left skew brace. That is, the near brace $\tT(B)$ and the left skew brace $(B,+_1,\circ)$ are isomorphic as near-trusses.
\end{corollary}

\begin{lemma} \label{lemmaa}
Let $(X, +, \circ)$ be a near brace and $z,w\in X$ satisfy the right distributivity.  
Consider also the maps $\sigma, \sigma' : X\times X \to X$ such that
$\sigma_a(b)= a\circ b -a \circ 0 \circ z +z $ and $\sigma'_a(b)= 
a\circ b -a \circ 0 \circ w +w.$ If $\sigma_a(b)= \sigma'_a(b)$ then $z^{-1}\circ w -1= w-z.$
\end{lemma}
\begin{proof} 
The proof is straightforward by setting $a = z^{-1}\circ 0^{-1}$ in both $\sigma_a(b)$ and $\sigma'_a(b).$ 
\end{proof}

\subsection{Generalized bijective maps $\&$ solutions of the braid equation}

Inspired by the findings of the preceding section we introduce below more general, multi-parametric bijective 
maps $\sigma^p_a,  \tau^p_b$  ($p$ stands for parametric) that provide solutions of the set-theoretic braid equation.

\begin{proposition}\label{lem:long2}
Let $(B,+, \circ)$ be a near  brace and let us denote $\sigma^p_a(b):=a\circ b\circ z_1-a \circ \xi + z_2$ 
and $\tau^{p}_b(a):=\sigma^p_{a}(b)^{-1}\circ a\circ b$, where $a,b\in B$,
and $h\in \{\xi,z_i\} \in B,$ $i \in \{1,2\}$ are fixed parameters,  such that there exist $c_1,c_2 \in B$ such that for all $ a,b,c \in B,$ $(a-b+c)\circ h=a\circ h-b\circ h+c\circ h,$ 
$~a\circ z_2\circ z_1 - a\circ \xi =c_1$ and $- a\circ \xi  + a\circ z_1\circ z_2 = c_2.$ 

Then for all $ a,b,c\in B$ the following properties hold:
\begin{enumerate}
\item $\sigma^p_{a}(b)\circ\tau^p_{b}(a)= a\circ b.$\label{lem:long:eq:5b}
\item $\sigma^p_{a}(\sigma^p_{b}(c))=a\circ b\circ c\circ z_1 \circ z_1-a\circ b\circ \xi \circ z_1 + c_1 + z_2 $.\label{lem:long:eq:2b}
\item $\sigma^p_{a}(b) \circ \sigma^p_{\tau^p_{b}(a)}(c)= a \circ b \circ c \circ z_1 +c_2  - a \circ \xi \circ z_2 +z_2\circ z_2.$
\label{lem:long:eq:6b}
\end{enumerate}
\end{proposition}
\begin{proof}
Let $a,b,c \in B$, then:

\begin{enumerate} 

\item $\szp{a}{b}\circ\tzp{b}{a}=\szp{a}{b}\circ \szp{a}{b}^{-1}\circ a\circ b= a\circ b.$

$ $

\item To show condition (2) we recall that $a\circ z_2\circ z_1 - a\circ \xi =c_1.$ Then,

\begin{eqnarray}
\szp{a}{\szp{b}{c}}&=&\szp{a}{b\circ c\circ z_1- b\circ \xi + z_2}\nonumber\\
&=&a\circ (b\circ c\circ z_1-b\circ \xi+  z_2)\circ z_1-a\circ \xi +z_2\nonumber\\ 
&=&a\circ b\circ c\circ z_1 \circ z_1-a\circ b\circ \xi \circ z_1+a\circ z_2 \circ z_1-a\circ \xi + z_2 \nonumber \\ 
&=&a\circ b\circ c\circ z_1 \circ z_1-a\circ b\circ \xi \circ z_1 + c_1 + z_2. \nonumber 
\end{eqnarray}

$ $

\item To show  condition (3) we use  (1) and $ - a\circ \xi  + a\circ z_1\circ z_2 = c_2.$

$$
\begin{aligned}
\sigma^p_{a}(b) \circ \sigma^p_{\tau^p_{b}(a)}(c)&=\szp{a}{b}\circ(\tzp{b}{a}\circ c \circ z_1-
 \tzp{b}{a} \circ \xi+z_2)\\ &=\szp{a}{b}\circ \tzp{a}{b}\circ c \circ z_1- \szp{a}{b}\circ \tzp{a}{b}\circ \xi+
\szp{a}{b}\circ z_2 \\ 
&= a\circ b \circ c    \circ z_1-a\circ b \circ \xi+ \szp{a}{b}\circ z_2\\ &=a\circ b\circ c \circ z_1-a\circ b\circ \xi+ 
(a\circ b\circ z_1-a\circ \xi +z_2)\circ z_2\\ &=a \circ b \circ c \circ z_1 +c_2  - a \circ \xi \circ z_2 +z_2\circ z_2.
\end{aligned} 
$$
\hfill \qedhere
\end{enumerate}
\end{proof}

\begin{example} 
A simple example of the above  generic maps is  the case where $z_1 \circ z_2 = \xi\circ 0 = z_2 \circ z_1, $ then $c_1 = c_2 =1.$  
\end{example}
\noindent Having showed the fundamental properties above we may now proceed in proving the following  theorem.

\begin{theorem}\label{prop:main}
Let $(B, +,\circ)$ be a near brace and $z\in B$ such that there exist $ c_1,c_2\in B$ such that for all $a,b,c\in B$, $(a-b+c)\circ z_i=a\circ z_i-b\circ z_i+c\circ z_i,$ $i \in \{1,2\},$ $a\circ z_2\circ z_1 - a\circ \xi =c_1$ and $- a\circ \xi  + a\circ z_1\circ z_2 = c_2.$
We define a map $\check{r}:B\times B\to B\times B$ given by
$$
\check{r}(a,b)=(\szp{a}{b},\tzp{b}{a}),
$$
where $\szp{a}{b}=a\circ b\circ z_1-a\circ \xi+z_2, $  $\tzp{b}{a}= \szp{a}{b}^{-1} \circ a \circ b.$ The pair $(B,\check{r})$ 
is a solution of the braid equation.
\end{theorem}
\begin{proof}
To prove this we need to show that the maps $\sigma, \tau$ satisfy the constraints (\ref{C1})-(\ref{C3}).  
To achieve this we use the properties proven in Proposition \ref{lem:long2}. 

Indeed, from Proposition \ref{lem:long2}, (1) and (2), it follows that (\ref{C1}) is satisfied, i.e. 
\[\szp{\eta}{\szp{x}{y}}=\szp{\szp{\eta}{x}}{\szp{\tzp{x}{\eta}}{y}}.\]

We observe that
\begin{equation}
\tau^p_{b}(\tau^p_a(\eta)) = T^p\circ \tau^p_{a}(\eta)\circ b = T^p \circ t^p \circ \eta \circ a \circ b = 
T^p \circ t^p \circ \eta \circ \szp{a}{b} \circ \tzp{b}{a}, \nonumber
\end{equation}
where $T^p= \szp{\tz{a}{\eta}}{b}^{-1}$ and $t^p= \sigma^p_{\eta}(a)^{-1} $ (the inverse in the  circle group). 
Due to (1),  (2),  (3)  of Proposition \ref{lem:long2} we then conclude that 
\[\tau^p_{b}(\tau^p_a(\eta)) =\tau^p_{\tzp{b}{a}}(\tau^p_{\szp{a}{b}}(\eta)),\] 
so (\ref{C2}) is also satisfied. 

To prove (\ref{C3}),  we employ (3),  (1) of Proposition \ref{lem:long2} and use the definition of $\tau^p$, 
$$
\sigma^p_{\tau^p_{\sigma^p_x(y)}(\eta)}(\tau^p_{y}(x))=\szp{\eta}{\sigma_x^p(y)}^{-1} \circ\sigma^p_{\eta}(x) \circ 
\sigma^p_{\tau^p_{x}(\eta)}(y)=\tzp{\szp{\tzp{x}{\eta}}{y}}{\szp{\eta}{x}}.
$$
Thus, (\ref{C3}) is satisfied,  and $\check{r}(a,b)=(\szp{a}{b},\tzp{b}{a})$ is a solution of the braid equation.
\end{proof}

\begin{lemma} \label{lemmab}
Let $(B, +, \circ)$ be a near brace and $z,w\in B$ satisfy the right distributivity.  Consider also the multi-parametric maps $\sigma^p,  \sigma^{'p} : B\times B \to B$ as defined in Proposition \ref{lem:long2}, such that
$\sigma^p_a(b)= a\circ b\circ z_1 -a \circ \xi +z_2 $ and $\sigma^{'p}_a(b)= a\circ b\circ z_2 -a \circ \xi +z_1.$ If $\sigma_a(b)= \sigma'_a(b)$ then $0\circ z_1^{-1}\circ z_2= z_2-z_1.$
\end{lemma}
\begin{proof}
The proof is straightforward by setting $a =0\circ  \xi^{-1}$ and $b=\xi\circ z_1^{-1}$ in both $\sigma^p_a(b)$ and $\sigma^{'p}_a(b).$
\end{proof}

\begin{remark} In the special case where  $z_1=1$ and $\xi = z_2 =z$ we recover the 
$\sigma^z_x(y), \ \tau_y^z(x)$ bijective maps and the $\check r_z$ solutions of the braid equation introduced in \cite{DoiRyb}.
\end{remark}

\begin{example}
We consider the brace $\qQ(O(i))$ from Example \ref{ex:complex}. Then, we can have  for instance the following choice of parameters:

\begin{enumerate}
\item $z_1=z_2=i,\  c_1 =c_2=i,\ \xi=-1, $ then $\sigma^p_a(b)=a\circ b\circ i+_ia$\\
\item $z_1=i,\  z_2=-i,\ c_1 =c_2=i,\ \xi=1, $ then $\sigma^p_a(b)=a\circ b\circ i-_ia$\\
\item $z_1=5,\  z_2=3,\  c_1 =c_2=i,\ \xi=15, $ then $\sigma^p_a(b)=a\circ b\circ 5-_i15\circ a+_i3.$
\end{enumerate}

\end{example}

In the following Proposition we provide the explicit expressions of the inverse $\check r$-matrices as 
well as the corresponding bijective maps.
\begin{proposition} \label{rem3} Let $\check r^*,\check r: B \times B \to B \times B$ 
such that $\check r^*: (x,y )\mapsto (\hat \sigma^{p}_x(y),  \hat \tau^{p}_y(x)),$ 
$\check r: (x, y) \mapsto (\sigma^{p}_x(y),   \tau^{p}_y(x))$ be solutions of the braid equation. Then the following statements hold.

\begin{enumerate} 
\item $\check r^* = \check r^{-1}$ if and only if 
\begin{eqnarray}
\hat \sigma^{p}_{\sigma_x^p(y)}(\tau_y^p(x)) =x, \  \hat \tau^{p}_{\tau^p_{y}(x)}(\sigma^p_{x}(y)) = 
y \ \mbox{and} \ \sigma^{p}_{\hat \sigma_x^{p}(y)}(\hat \tau_y^{p}(x)) =x,  \ \tau^{p}_{\hat \tau^{p}_{y}(x)}(\hat \sigma^{p}_{x}(y)) = y.   \label{k2}
\end{eqnarray}

\item Let $\sigma^p_x(y) = x\circ y \circ z_1- x\circ \xi +z_2,$ $\ \tau^p_y(x) = \sigma^p_x(y)^{-1}\circ x \circ y,$ 
and $\xi, z_i\in B$ $i \in \{1,2\}$ are fixed elements,  such that there exists $c_1,c_2 \in B$ such that for all  $a,b,c \in B,$ $(a-b+c)\circ z_i=a\circ z_i-b\circ z_i+c\circ z_i,$ 
$a\circ z_2\circ z_1 - a\circ \xi =c_1$ and $ - a\circ \xi  + a\circ z_1\circ z_2 = c_2.$Then $\hat\sigma^{p}_x(y)= \hat z_2 -x \circ \hat \xi + x\circ y \circ \hat z_1, $ 
$\ \hat \tau^p_x(y) =\hat\sigma^{p}_x(y)^{-1}\circ x \circ y,$ where $\hat \xi = \xi^{-1}, $ 
$\hat z_{1,2} = z_{1,2} \circ \xi^{-1}.$ 
\end{enumerate}
\end{proposition}
\begin{proof} We prove the two parts of Proposition \ref{rem3} below:
\begin{enumerate}
\item If $\check r^* = \check r ^{-1},$ then $\check r \check r^* = \check r^* \check 
r= \mbox{id}$ and 
\begin{equation}
\check r \check r^*(x, y) = ( \sigma^p_{\hat \sigma_x(y)}(\hat \tau^p_{y}(x)),   
\tau^{p}_{\hat \tau^p_{y}(x)}(\hat \sigma^p_{x}(y))). \nonumber
\end{equation}
Thus  
\begin{equation}
\sigma^{z}_{\hat \sigma_x^{p}(y)}
(\hat \tau_y^{p}(x)) =x, \quad \tau^{p}_{\hat \tau^{p}_{y}(x)}(\hat \sigma^{ z}_{x}(y)) = y. \nonumber
\end{equation} 
And vice versa if $\sigma^{p}_{\hat \sigma_x^{p}(y)}(\hat \tau_y^{p}(x)) =
x, $\ $ \tau^{p}_{\hat \tau^{p}_{y}(x)}(\hat \sigma^{p}_{x}(y)) = y,$ 
then it automatically follows that $\check r^{*} = \check r^{-1}.$ 

Similarly,  $\check r^* \check r(x,y)= (x, y)$ leads to 
$\hat \sigma^{p}_{\sigma_x^p(y)}(\tau_y^p(x)) =x, $ $\hat \tau^{p}_{\tau^p_{y}(x)}(\sigma^p_{x}(y)) = y,$ 
and vice versa.

\item For the second part of the Proposition it suffices to show (\ref{k2}). Indeed, we recall that that $ \hat \xi = \xi^{-1}$ 
and $\hat z_{1,2} =z_{1,2} \circ \xi^{-1},$ then
\begin{eqnarray}
 \hat \sigma^p_{\sigma^p_x(y)}(\tau^p_y(x)) &=& \hat z_2 -\sigma^p_x(y) \circ \hat \xi+ \sigma^p_x(y)\circ \tau_y^p(x) \circ \hat  z_1 \nonumber\\
&=&  \hat z_2
-\sigma^p_x(y) \circ \hat \xi + x \circ y \circ \hat z_2 \nonumber\\
&=& \hat z_2 - (x\circ y \circ z_1 - x\circ \xi +z_2)\circ \hat \xi +  x \circ y \circ \hat z_1 = x. \nonumber
\end{eqnarray}
Also, $\hat \tau^{p}_{\tau^p_{y}(x)}(\sigma^p_{x}(y))  = x^{-1} \circ \sigma^p_x(y) \circ \tau_y^p(x)= y. $

Similarly,  we show
\begin{eqnarray}
\sigma^{p}_{\hat \sigma_x^{p}(y)}(\hat \tau_y^{p}(x))  &=& 
 \hat \sigma^p_x(y)\circ \hat \tau_y^p(x)\circ z_1 -  \hat \sigma^p_x(y) \circ \xi +z_2 \nonumber\\
&=& x \circ y \circ z_1 - (\hat z_2 -x\circ \hat \xi +x\circ y \circ \hat  z_1)\circ \xi
 +z_2  = x.\nonumber
\end{eqnarray}
And as above we immediately deduce that $\tau^{p}_{\hat \tau^{p}_{y}(x)}(\hat \sigma^{p}_{x}(y)) = y.$
\hfill \qedhere
\end{enumerate}
\end{proof}

With this we conclude our analysis on the general bijective maps coming from near  
braces and the corresponding solutions of the braid equation.

\subsection{$p$-deformed braided groups and near braces}
\noindent Motivated by the definition of braided groups and braidings in \cite{chin} 
as well as the relevant work presented in \cite{GatMaj} we provide a generic definition of the $p$-deformed braided group and braiding
that contain extra fixed parameters, i.e.  for multi-parametric braidings ($p$-braidings).

\begin{definition} \label{defp}
Let $(G,\circ)$ be a group,  $m(x,y) = x \circ y$ and $\check r$ is an invertible map $\check r: G \times G \to G \times G,$ such that for all $ x, y \in G, $ $\check r(x,y) = (\sigma^p_x(y), \tau^p_y(x)),$ where $\sigma^p_x,\ \tau^p_y$ 
are bijective maps in $G$.  The map $\check r$ is called  a $p$-braiding operator 
(and the group is called $p$-braided) if
\begin{enumerate}
\item $x\circ y = \sigma^p_x(y) \circ \tau^p_y(x).$
\item $  (\mbox{id} \times m)\ \check r_{12}\ \check r_{23} (x,y,w)= (f^p_{x\circ y}(w), \ f^p_{x\circ y}(w)^{-1}\circ x \circ y\circ w).$
\item $  (m \times \mbox{id})\ \check r_{23}\ \check r_{12}(x,y,w)= (g^p_x(y\circ w), \  g^p_x(y\circ w)^{-1}\circ x \circ y\circ w).$
\end{enumerate}
for some bijections $f_x^p, g_x^p: G \to G,$ given for all $x \in G.$
\end{definition}

\begin{proposition} \label{prop2} 
Let $(G,\circ)$ be a group, and the invertible map $\check r: G \times G \to G \times G,$ $\check r(x,y) = 
(\sigma^p_x(y), \tau^p_y(x)),$  be a $p$-braiding operator for the group $G.$ Then $\check r$ 
is a non-degenerate solution of the braid equation.
\end{proposition}
\begin{proof} We start from the LHS of condition (2) of  Proposition \ref{prop2}
\begin{eqnarray}
 (\mbox{id} \times m)\ \check r_{12}\ \check r_{23} \ (x,y,w) = (\sigma^p_x(\sigma^p_y(w)), \ 
 \tau^p_{\sigma^p_y(w)}(x) \circ \tau^p_w(y)), \nonumber
\end{eqnarray}
which leads to 
\begin{equation}
\sigma^p_x(\sigma^p_y(w))  = f^p_{x\circ y}(w) = f^p_{\sigma^p_x(y)\circ \tau^p_y(x)}(w)  = 
\sigma^p_{\sigma^p_x(y)}(\sigma^p_{\tau^p_y(x)}(w)) \nonumber
\end{equation}
i.e.  the fundamental condition (\ref{C1}) is satisfied.  Moreover,  using condition (1) we show
\begin{eqnarray}
\tau^p_{\sigma^p_y(w)}(x) \circ \tau^p_w(y) = \sigma^p_x(\sigma^p_y(w))^{-1}\circ x\circ \sigma^p_y(w) \circ 
\sigma^p_y(w)^{-1}\circ y \circ w
=  f^p_{x\circ y}(w)^{-1}\circ x \circ y\circ w, \nonumber
\end{eqnarray}
as expected compatible with condition (2) of Proposition \ref{prop2}.

Similarly,  from the LHS of condition (3) 
\begin{eqnarray}
(m \times \mbox{id})\ \check r_{23}\ \check r_{12}\ (x,y,w)= (\sigma_x^p(y) \circ \sigma^p_{\tau^p_y(x)}(w),\  
\tau^p_w(\tau^p_y(x))). \nonumber
\end{eqnarray}
The latter expression leads to 
\begin{equation}
\sigma_x^p(y) \circ \sigma^p_{\tau^p_y(x)}(w) = g_x^p(y\circ w).
\end{equation}
Also, via condition (1) 
\begin{eqnarray}
\tau^p_w(\tau^p_y(x)) &=& \sigma^p_{\tau^p_y(x)}(w)^{-1} \circ \sigma_x^p(y)^{-1} \circ x \circ y \circ w 
= g^p_x(y\circ w)^{-1} \circ x \circ y\circ w \nonumber\\
 &=&\tau^p_{\tau^p_w(y)}(\tau^p_{\sigma^p_y(w)}(x))
\end{eqnarray}
compatible with condition (3) of Proposition \ref{prop2},  and this shows condition (\ref{C3}).

Having shown properties \ref{C1} and (\ref{C2}) and taking into account that $x\circ y = \sigma_x^p(y) \circ \tau_y^p(x)$ 
we also show condition (\ref{C3}),  i.e.  we conclude that $\check r,$ as defined in Proposition \ref{prop2},  
is a solution of the braid equation.
\end{proof}

\begin{lemma} Let  $B$ be a near brace, and consider the map $\check r: B \times B \to B\times B,$ 
$\check r(x,y) = ( \sigma^p_x(y), \ \tau^p_y(x))$ of Proposition \ref{prop:main}. Then $\check r$ is a $p$-braiding.
\end{lemma}
\begin{proof}
The proof is straightfroward via Proposition \ref{lem:long2}. Indeed, all the conditions of the $p$-braiding Definition \ref{defp} 
are satisfied and:\\
 $f^p_a(b) = a\circ b \circ z_1\circ z_1 - a\circ\xi\circ z_1 + c_1 +z_2,$ $~ g^p_a(b) = a\circ b \circ z_1 +c_2-a\circ \xi\circ z_2+z_2\circ z_2.$
\end{proof}

With this we conclude our analysis on $p$-braidings and their connection to the YBE and the notion of the near  braces. One of the fundamental open problems in this frame and a natural next step is the solution of the set-theoretic reflection equation for this new class of solutions of the set-theoretic YBE.
We hope to address this problem and generalize the notion of the $p$-braiding to include the reflection equation, in the near future.  Another key question, which we hope to tackle soon,  is what the effect of non-associativity in $(X, +)$ on the construction of the algebraic structures emerging from solutions of the set-theoretic YBE would be. This is quite a challenging problem, the analysis of which will yield yet more generalized classes of solutions of the YBE.

\subsection*{Acknowledgments}
\noindent  Support from the EPSRC research grant EP/V008129/1 is acknowledged.

\end{document}